\documentclass[3p]{elsarticle}

\usepackage{amsmath}
\usepackage{amsfonts, euscript}
\usepackage{amssymb}
\usepackage{amscd}
\usepackage{amsthm}
\usepackage{mathrsfs}
\usepackage{graphicx}
\usepackage[all]{xy}
\usepackage{pstricks, pstricks-add, pst-math, pst-xkey}
\usepackage{epsfig}
\usepackage{color}
\usepackage{calc}
\usepackage{hyperref}
\usepackage{hyphenat}

\usepackage[T1]{fontenc}
\usepackage{times}

\newtheorem{thm}{Theorem}
\newtheorem{cor}[thm]{Corollary}
\newtheorem{claim}[thm]{Claim}
\newtheorem{prop}[thm]{Proposition}
\newtheorem{lem}[thm]{Lemma}
\newtheorem{obs}[thm]{Observation}
\newtheorem{rem}[thm]{Remark}
\newtheorem{conj}[thm]{Conjecture}

\newcommand{\mc}{\mathcal}

\def\rcr{\overline{\hbox{\rm cr}}}
\def\pcr{\widetilde{\hbox{\rm cr}}}
\def\hal{\hbox{\rm hal}}

\journal{Discrete Mathematics}

\begin{document}

\begin{frontmatter}

\title{Point sets that minimize $(\le k)$-edges, $3$-decomposable drawings, and the rectilinear crossing number of $K_{30}$}

\author[chv]{M.~Cetina}
\ead{mcetina@ifisica.uaslp.mx}

\author[chv]{C.~Hern\'andez--V\'elez}
\ead{cesar@ifisica.uaslp.mx}

\author[jlm]{J.~Lea\~nos}
\ead{jleanos@mate.reduaz.mx}

\author[chv]{C.~Villalobos}

\address[chv]{Instituto de F\'{\i}sica, Universidad Aut\'onoma de San Luis Potos\'{\i}, San Luis Potos\'{\i}, M\'exico 78000}

\address[jlm]{Unidad Acad\'emica de Matem\'aticas, Universidad Aut\'onoma de Zacatecas, Zacatecas, M\'exico 98060}

\begin{abstract}
There are two properties shared by all known crossing-minimizing geometric drawings of $K_n$, for $n$ a multiple of $3$. First, the underlying $n$-point set of these drawings minimizes the number of $(\le k)$-edges, that means, has exactly $3\binom{k+2}{2}$ $(\le k)$-edges, for all $0\le k < n/3$. Second, all such drawings have the $n$ points divided into three groups of equal size; this last property is captured under the concept of $3$-decomposability.  In this paper we show that these properties are tightly related:   every $n$-point set with exactly $3\binom{k+2}{2}$ $(\le k)$-edges for all $0\le k < n/3$, is $3$-decomposable.  The converse, however, is easy to see that it is false. As an application, we prove that the rectilinear crossing number of $K_{30}$ is $9726$.
\end{abstract}


\end{frontmatter}

\section{Introduction}

The {\em rectilinear crossing number} $\rcr(G)$ of a graph $G$, is the minimum number of edge crossings in a {\em geometric drawing} of $G$ in the plane, that is, a drawing of $G$ in the plane where the vertices are points in general position and the edges are straight segments. Determining $\rcr(K_n)$, where $K_n$ is the complete graph with $n$ vertices, is a well-know open problem in combinatorial geometry initiated by Guy~\cite{Gu}.

The rectilinear crossing number problem is related with the concept of $k$-edges. A {\em $k$--edge} of an $n$--point set $P$, with $0 \le k \le n/2 - 1$, is a line through two points of $P$ leaving exactly $k$ points on one side. A $(\leq k)$--edge is an $i$--edge with $0 \leq i \leq k$. Let $E_k(P)$ denote the number of $k$--edges of $P$ and $E_{\le k}(P)$ denotes the number of $(\le k)$-edges, that is, $E_{\le k}(P) = \sum_{j=0}^k E_j(P)$. Finally, $E_{\le k}(n)$ denotes the minimum of $E_{\le k}(P)$ taken over all $n$-point sets $P$.

The exact determination of $E_{\le k}(n)$ is another notable open problem in combinatorial geometry. In 2005~\cite{AiGa}, Aichholzer et al. gave the following lower bound for $E_{\le k}(n)$:

\begin{equation}
E_{\le k}(n) \ge 3\binom{k+2}{2} + 3\binom{k+2-\lfloor n/3 \rfloor}{2} - \max\{0,(k+1-\lfloor n/3 \rfloor)(n - 3\lfloor n/3 \rfloor)\}\label{eq:edg},
\end{equation}
later, in 2007~\cite{AiGa07}, Aichholzer et al. proved that this lower bound is tight for $k \le \lfloor 5n/12 \rfloor -1$. 

The number of crossings in a geometric drawing of $K_n$ and the number of $k$-- and $(\le k)$--edges in the underlying $n$-point set $P$ are closely related by the following equality, independently proved by L\'ovasz et al.~\cite{LoVe} and \'Abrego and Fern\'andez-Merchant~\cite{Ab1}. For any set $P$ of $n$ points

\begin{align}
\rcr(P)  &  =3\binom{n}{4}-
{\displaystyle\sum\limits_{k=0}^{\left\lfloor n/2\right\rfloor -1}}
k\left(  n-k-2\right)  E_{k}\left(  P\right)  ,\text{ or equivalently,}
\nonumber\\
\rcr(P)  &  =
\left({\displaystyle\sum\limits_{k=0}^{\left\lfloor n/2\right\rfloor -1}}
\left(  n-2k-3\right)  E_{\leq k}\left(  P\right)\right)  -\frac{3}{4}\binom{n}
{3}+\left(  1+\left(  -1\right)  ^{n+1}\right)  \frac{1}{8}\binom{n}{2}.
\label{eq:credg}
\end{align}

Another concept that plays a central role in this paper is the $3$--decomposability, which is a  property shared by all known crossing-minimizing geometric drawings of $K_n$, for $n$ a multiple of $3$. Formally, we say that a finite point set $P$ is {\em $3$--decomposable} if it can be partitioned into three equal sized sets $A, B$ and $C$ such that there exists a triangle $T$ enclosing the point set $P$ and the orthogonal projection of $P$ onto the three sides of $T$ show $A$ between $B$ and $C$ on one side, $B$ between $C$ and $A$ on the second side, and $C$ between $A$ and $B$ on the third side. We say that a geometric drawing of $K_n$ is $3$-decomposable if its underlying point set is $3$-decomposable.

In the following result we establish the relationship between $3$--decomposability and the number of $(\le k)$-edges

\begin{thm}[Main Theorem]\label{thm:main}
Let $P$ be an $n$--point set, for $n$ a multiple of $3$,  with exactly $3\binom{k+2}{2}$ $(\le k)$-edges for all $0 \le k < n/3$, then $P$ is 3--decomposable.
\end{thm}

In fact, in~\cite{Ab:3sim} \'Abrego et al. conjectured that for each positive integer $n$ multiple of $3$, all crossing--minimal geometric drawings of $K_n$ are $3$--decomposable.

As an application of the Main Theorem we prove that a $30$-point set that minimize the crossing number is $3$-decomposable. Aichholzer established $9726$ as the upper bound for $\rcr(K_{30})$~\cite{Aiweb}, moreover we have the following theorem

\begin{thm}[The Rectilinear Crossing Number of $K_{30}$]\label{thm:second}
$\rcr(K_{30})$ is $9726$.
\end{thm}

All the results of this paper are proved in the more general context of generalized configuration of points. 
In this scope we define by analogy the {\em pseudolinear crossing number} $\pcr(K_n)$. 

Our main tool are the allowable sequences which will be formally define in Section~\ref{sec:allseq}, and we mention some preliminary results due to L\'ovasz et al. in~\cite{LoVe}. In Section~\ref{sec:mainthm} we prove the Main Theorem. In Section~\ref{sec:k30} we use the Main Theorem to establish that a configuration with $30$ points that minimize the crossing number is $3$--decomposable and we give some implications of the $3$--decomposability. Finally, in Section~\ref{sec:proofk30} is the formal proof of Theorem~\ref{thm:second}.

\section{Allowable Sequences}\label{sec:allseq}

An {\em allowable sequence $\bf \Pi$} is a doubly infinite sequence $\ldots, \pi_{-1}, \pi_0, \pi_1, \ldots$ of permutations of $n$ elements, where consecutive permutations differ by a transposition of neighboring elements, and $\pi_i$ is the reverse permutation of $\pi_{i+\binom{n}{2}}$. Thus $\bf \Pi$ has period $2\binom{n}{2}$, and the hole information of $\bf \Pi$ is contained in any of its $n$--half-periods, which we call {\em $n$--half-periods}. We usually denote by $\Pi$ an $n$-half-period of $\bf \Pi$. 

It is know that if $P$ is a set of $n$ points in the plane in general position, then all the combinatorial information of $P$ can be encoded by an allowable sequence $\bf \Pi_P$ on the set $P$, called
{\em circular sequence}  associated to $P$~\cite{GoPo}. It is important to note that most allowable sequence are not circular sequences, however there is a one-to-one correspondence between allowable sequences and generalized configurations of points~\cite{GoPo}. 

We have the following definitions and notations for allowable sequences. A transposition that occurs between elements in sites $i$ and $i + 1$ is an {\em $i$--transposition}, and we say that moves through the {\em $i$th--gate}. In this new setting an $i$--transposition, or $(n-i)$--transposition corresponds to an $(i-1)$--edge. For $i \leq n/2$, an {\em $i$--critical transposition} is either an $i$--transposition or an $(n-i)$--transposition, and a {\em $(\leq k)$--critical transposition} is a transposition that is $i$--critical for some $1 \leq i \leq k$. If $\Pi$ is an $n$--half-period, then $N_k(\Pi)$ and $N_{\leq k}(\Pi)$ denote the number of $k$-critical transpositions and $(\leq k)$--critical transpositions in $\Pi$, respectively. Therefore $N_k(\Pi)=E_{k-1}(\Pi), N_{\le k}(\Pi)=E_{\le k-1}(\Pi)$. When $n$ is even an $n/2$--transposition is also called {\em halving} and $h(\Pi)$ denotes the number of halvings, and thus $h(\Pi) = E_{n/2 - 1}(\Pi)$.

The identity (\ref{eq:credg}) relating $k$-edges to
crossing number was originally proved for allowable sequences. All these definitions
and functions coincide with their original counterparts for $P$ when $\Pi$ is
the circular sequence of $P$. However, when $\rcr(n),$ and $E_{\leq k}\left(  n\right)  $ are minimized 
over all allowable sequences on $n$ points rather than over all sets
of $n$ points, the corresponding quantities may change so we define the notation $\widetilde{cr}(n)$ and $\widetilde{E}_{\leq
k}\left(  n\right)  $. But it is clear that $\pcr(n) \leq \rcr(n)$ and $\widetilde
{E}_{\leq k}\left(  n\right)  \leq E_{\leq k}\left(  n\right)$. \'{A}brego
et al. \cite{Ab2} proved that the lower bound~(\ref{eq:edg}) on $E_{\leq
k}\left(  n\right)  $ is also a lower bound on $\widetilde{E}_{\leq k}\left(
n\right)  $ and use it to extend the lower bound on $\overline{cr}(n)$ to $\widetilde{cr}(n)$.

Let $\Pi=(\pi_0,\pi_1, \ldots, \pi_{\binom{n}{2}})$ be an $n$--half-period.
For each $k < n/2$, define $m=m(k,n):=n-2k$. In order to
keep track of $(\le k)$--critical transpositions in $\Pi$, it is convenient to
label the points so that the starting permutation is
$$
\pi_0 = (a_k, a_{k-1}, \ldots, a_1, b_1, b_2, \ldots, b_m, c_1, c_2,
\ldots, c_k).
$$

Sometimes it will be necessary to say who an element is moving, so we will say that an element $x$ {\em exits} (respectively, {\em enters}) {\em through
  the $i$th $A$--gate} if it moves from the position $k-i+1$ to the position
  $k-i+2$ (respectively, from the position $k-i+2$ to the position $k-i+1$)
  during a transposition with another element. Similarly, $x$ {\em
  exits} (respectively, {\em enters}) {\em through the $i$th
  $C$--gate} if it moves from the position $m+k+i$ to the position $m+k+i-1$
  (respectively, from $m+k+i-1$ to $m+k+i$) during a transposition.

An $a\in \{a_1, \ldots, a_k\}$ (respectively, $c\in \{c_1,\ldots,
c_k\}$) is {\em confined} until the first time it exits through the
first $A$--gate (respectively, $C$--gate); then it becomes {\em free}.
A transposition is {\em confined} if both elements involved are
confined.

The following results, from Proposition~\ref{prop:noconf} to Proposition~\ref{prop:perfect}, are due to Lov\'asz et al. in the paper~\cite{LoVe}:

\begin{prop}\label{prop:noconf}
Let $\Pi_0$ be an $n$--half-period, and let $k < n/2$. Then there
is an $n$--half-period $\Pi$, with the same number of $(\le
k)$--critical transpositions as $\Pi_0$, and with no confined
transpositions.
\end{prop}

In view of this statement, for the rest of the section we assume that
the $n$--half-period $\Pi$ under consideration has no confined
transpositions.

The {\em liberation sequence} $\sigma(\Pi)$ (or simply $\sigma$ if no
confusion arises) of $\Pi$ contains all the
$a$'s and all the $c$'s, in the order in which they become free in
$\Pi$. Since $\Pi$ has no confined transpositions, the $a$'s appear in
increasing order, as do the $c$'s.  We let $T(a_i)$ (respectively
$T(c_i)$) denote the set of all those $c$'s (respectively $a$'s)
that appear after $a_i$ (respectively $c_i$) in $\sigma$.

A transposition that swaps elements in the positions $i$ and $i+1$ occurs
  in the {\em A--Zone} (respectively, {\em C--Zone}) if $i \le k$
  (respectively, $i \ge k+m$).  Such transpositions are of obvious
  relevance: a transposition is $(\le k)$--critical if and only if it occurs
  either in the $A$--Zone or in the $C$--Zone.

For $1 \le i \le j \le k$, the $i$th $A$--gate is a {\em compulsory
exit--gate} for $a_j$, and the $i$th $C$--gate is a {\em compulsory
entry--gate} for $a_j$: that is, $a_j$ has to exit through the $i$th
$A$--gate at least once, and to enter the $i$th $C$--gate at least
once.  Analogous definitions and observations hold for $c_j$: the
$i$th $A$--gate is a {\em compulsory entry--gate} for $c_j$, and the
$i$th $C$--gate is a {\em compulsory exit--gate} for $c_j$.  A
transposition in which an element enters (respectively, exits) one of
its compulsory entry (respectively, exit)--gates for the first time is
a {\em discovery} transposition {\em for} the element.  A
transposition is a {\em discovery} transposition if it is a discovery
transposition for at least one of the elements involved. If it is a
discovery transposition for both elements, then it is a {\em
double--discovery} transposition (for the reader familiar
with~\cite{LoVe}, what we call double--discovery transpositions are
the transpositions represented by a directed edge in the {\em savings
digraph} of~\cite{LoVe}).

Discovery and double--discovery transpositions play a central role
in~\cite{LoVe}.  The key results are the following, which hold for
any $n$--half-period with no confined transpositions (the first
statement is a straightforward counting, whereas the second does
definitely require a proof).

\begin{obs}\label{obs:disctrans}
There are (exactly) $2\binom{k+1}{2}$ transpositions that are discovery
transpositions for some $a$, and (exactly) $2\binom{k+1}{2}$
transpositions that are discovery transpositions for some $c$.
\end{obs}

\begin{prop}\label{prop:dobldisc}
There are {\em at most} $\binom{k+1}{2}$ double--discovery
transpositions.
\end{prop}

Since each discovery transposition is $(\le
k)$--critical, these statements immediately imply the following.

\begin{prop}\label{prop:lwcritrans}
There are at least $3\binom{k+1}{2}$ $(\le k)$--critical
transpositions.
\end{prop}

An $n$--half-period $\Pi$ with no confined transpositions
  is  {\it perfect} if the following
hold:
\begin{description}
\item{(a)} Each transposition in $\Pi$ that occurs in the $A$--Zone or
  in the $C$--Zone is a discovery transposition.

\item{(b)}  $a_i$ is involved in (exactly) $\min\{i,|T(a_i)|\}$
  double--discovery transpositions in the $C$--Zone.

\item{(c)} Each $c_i$ is involved in (exactly) $\min\{i,|T(c_i)|\}$
  double--discovery transpositions in the $A$--Zone.
\end{description}

The following result is implicit in the proof of Theorem 10 in~\cite{LoVe}.

\begin{prop}\label{prop:perfect}
If $\Pi$ is perfect, then it has exactly $3\binom{k+1}{2}$ $(\le
k)$--critical transpositions for all $k \leq m$. Conversely, if $\Pi$ has no confined
transpositions, and has exactly $3\binom{k+1}{2}$ $(\le k)$--critical
transpositions for all $k \leq m$, then it is perfect.
\end{prop}

\section{Proof of Main Theorem}\label{sec:mainthm}

The concept of $3$-decomposability for $n$-point sets is also generalized in the setting of allowable sequences. An $n$--half-period $\Pi$ of an allowable sequence $\bf\Pi$ is {\em 3--decomposable} if  the elements  in $\Pi$ can be labeled $A = \{a_{n/3}, a_{n/3-1}$, $\ldots, a_1\}$,  $B = \{b_{1}, b_{2}, \ldots, b_{n/3}\}$, $C = \{c_1,c_{2}, \ldots, c_{n/3}\}$ and if $\pi_0 = (a_{n/3}, a_{n/3-1}, \ldots, a_1$, $b_1,b_{2}, \ldots, b_{n/3}$, $c_1, c_{2}, \ldots, c_{n/3})$ is the first permutation of $\Pi$, thus, all transpositions between an element of $A$ and an element of $B$ occur before that the transpositions between $C$ and $A \cup B$, after occur all transposition between $A$ and $C$ prior that the transposition between $B$ and $ C$ and later occur all transposition between $C$ and $B$. In particular, there are some indices $0 < s < t < \binom{n}{2}$, such that $\pi_{s+1}$ shows all the $b$-elements followed by all the $a$-elements followed by all de $c$-elements, and $\pi_{t+1}$ shows all the $b$-elements followed by all the $c$-elements followed by all the $a$-elements. An  allowable sequence is {\em 3--decomposable} if contains an $n$--half-period $3$--decomposable.

Before proving the Main Theorem, we must first state two propositions:

\begin{prop}\label{prop:libseq}
Suppose that $\Pi$ is perfect. Then, in the liberation sequence
$\sigma$ of $\Pi$, either all the $a$'s occur consecutively or
all the $c$'s occur consecutively.
\end{prop}

\begin{proof}
The last entry in $\sigma$ is either $a_k$ or
$c_k$, and by symmetry
we may assume without any loss of generality that it
is $a_k$.  Our strategy
is to suppose that  $a_{t-1} c_\ell
c_{\ell+1} \cdots c_k a_t \cdots a_k$ is a suffix of $\sigma$,
where $\ell > 1$ and $2 \le t \le k$, and derive a contradiction.

We claim that $a_{t-1}$ swaps with $c_k$ in the $C$--Zone.  We start
by noting that since $\Pi$ is perfect, and $|T(a_{t-1})| = k - \ell +1
\ge 1$, it follows that $a_{t-1}$ is involved in a double--discovery
transposition in the $C$--Zone with at least one $c$. If this
transposition involves ($a_{t-1}$ and) $c_k$, then our claim obviously
holds. Thus suppose that it involves ($a_{t-1}$ and) $c_i$ for some
$i< k$. Then, right after $a_{t-1}$ and $c_i$ swap, $c_k$ is to the
right of $a_{t-1}$, since no confined transpositions occur in
$\Pi$. Note that all transpositions that swap $a_{t-1}$ to the left
involve an $a_j$ with $j > t-1$. On the other hand, since $a_t$
(moreover, every $a_j$ with $j \ge t$) gets freed after $c_k$, it
follows that before any transposition can move $a_{t-1}$ left, $c_k$
must be freed (and before that it must transpose with $a_{t-1}$). This shows that the
transposition $\mu$ that swaps $a_{t-1}$ with $c_k$ occurs in the
$C$--Zone.

Thus, right after $\mu$ occurs, $a_{t-1}$ is at position
$r$, where $r \ge k+m+1$. We claim that $\max\{r,k+m+t-1\} <
2k+m$. Since $t-1 < k$, then $k+m+t-1 < 2k+m$, and so it suffices to
show that if $r > k+m+t-1$, then $r < 2k+m$.  So suppose that $r >
k+m+t-1$. Note that the final position in $\Pi$ (that is, the position in
$\pi_{\binom{n}{2}}$) of $a_{t-1}$ is $k+m+t-1$, and so by the time
$\mu$ occurs there has been a
transposition $\tau$ that moves $a_{t-1}$ to the right of
its final position (we remark that possibly $\tau = \mu$). Since $\tau$
occurs in the $C$--Zone and clearly is not a discovery step for
$a_{t-1}$, and $\Pi$ is perfect, it follows that $\tau$ is a discovery step for a
$c_i$. Moreover,  $|T(a_{t-1})| = k - \ell+1$ is
greater than $t-1$, as otherwise (by the perfectness of $\Pi$) the
transposition between $a_{t-1}$ and $c_i$ would have to be
a double--discovery step. Thus $|T(a_{t-1})| > t-1$, and again
invoking the perfectness of $\Pi$ we get that $a_{t-1}$ is involved
with (exactly) $t-1$ double--discovery steps in the $C$--Zone,
each with an element in
$\{c_\ell,\ldots,c_{k}\}$. Therefore the number of
possible transpositions that move $a_{t-1}$ to the right of its final position
$k+m+t-1$ is at most $k-\ell+1-(t-1)$.  Thus the rightmost position of
$a_{t-1}$ throughout $\Pi$ (and consequently $r$) is at most $k+m+t-1 + k-\ell+1-(t-1)
= 2k+m +1 -\ell < 2k+m$.

Let $R$ be the set of the points that occupy the positions $r+1, r+2,
\ldots, 2k+m$ immediately after $\mu$ occurs.  Since at this time
every $a_j$ with $j > t-1$ is confined, it follows that each point in
$R$ is either a $b$, a free $c$ (this follows easily since there are
no confined transpositions, and $a_{t-1}$ reached the position $r$ by
transposing with $c_k$), or an $a_j$ with $j < t-1$.  In
particular, each element in $R$ still has to transpose with $a_{t-1}$.

We claim that  $a_{t-1}$ must move back to the $B$--Zone (after
$\mu$ occurs).  Seeking a contradiction, suppose that
$a_{t-1}$ does not go back to the $B$--Zone.  We claim that then there is a transposition $\rho$ of $a_{t-1}$ with an
element in $R$ that is not a discovery transposition.  The key observation is
that then at most $k+m+t-1-r$ transpositions of $a_{t-1}$ with
elements of $R$ can be discovery transpositions. In order to prove
this assertion, first we note that no transposition of $a_{t-1}$ with
an element in $R$ can be discovery transposition for the element in
$R$ (recall that each element in $R$ is either a $b$, a free $c$, or
an $a_j$ with $j < t-1$), so if such a transposition is a discovery
one, it is so for $a_{t-1}$ (recall we assume that $a_{t-1}$ does not
go back to the $B$--Zone). But once $a_{t-1}$ has reached $r$, it has
at most $k+m+t-1-r$ discovery transpositions to do (since the rightmost
compulsory entry--gate for $a_{t-1}$ is the $(t-1)$st $C$--gate). Now
since $R$ has $2k+m-r$ elements, and $2k+m-r > k+m+t-1-r$, it follows
that there is at least one transposition $\rho$ of $a_{t-1}$ with an
element of $R$ that is {not} a discovery transposition, as
claimed. But the perfectness of $\Pi$ implies that such a
transposition must occur in the $B$--Zone, contradicting (precisely)
our assumption that $a_{t-1}$ did not move back to the $B$--Zone.

Thus, after $\mu$ occurs, $a_{t-1}$
eventually re-enters the $B$--Zone, and since its final
position is $k+m+t-1$, afterwards it has to re-enter the $C$--Zone via a
transposition $\lambda$ that moves $a_{t-1}$ to the right and an
element $x\in R$ to the left. Since $\lambda$ occurs in the $C$--Zone,
and $\Pi$ is perfect, then $\lambda$ must be a discovery
transposition. We complete the proof by arriving to a contradiction:
$\lambda$ {\em cannot} be a discovery transposition.  Indeed,
$\lambda$ cannot be discovery for $a_{t-1}$ (since it had already been
in the $C$--Zone), so it must be a discovery step for $x$. On the
other hand, since each $x\in R$ is either a $b$, a free $c$, or an
$a_j$ with $j < t-1$, $\lambda$ it follows that $\lambda$ cannot be a
discovery transposition for $x$ either.
\end{proof}

Our next statement shows that we can actually go a bit further: there is a perfect $n$--half-period $\Pi'$ whose liberation sequence has all $a$'s followed by all $c$'s or vice versa.

\begin{prop}\label{prop:3decomp}
Suppose that $\Pi$ is a perfect $n$--half-period of an allowable sequence $\bf \Pi$. Then $\bf \Pi$ contains a perfect $n$--half-period $\Pi'$, with initial permutation $a_k' a_{k-1}' \ldots a_1'$ $b_1' \ldots b_m'$ $c_1' c_2' \ldots c_k'$, and whose liberation sequence is either $a_1' a_{2}' \ldots a_k' c_1' c_2' \ldots  c_k'$ or $c_1' c_{2}' \ldots c_k' a_1' a_2' \ldots  a_k'$.
\end{prop}

\begin{proof}
Let $\Pi=(\pi_0, \pi_1, \ldots,
\pi_{\binom{n}{2}})$ be any perfect $n$--half-period, with initial
permutation $\pi_0=(a_k a_{k-1} \ldots a_1$ $ b_1 \ldots b_m c_1 c_2
\ldots c_k)$, and let  $\sigma$ be the liberation sequence associated to
$\Pi$. Thus the last entry of $\sigma$ is either $a_k$ or $c_k$, and a
straightforward symmetry argument shows that we may assume without
loss of generality that last entry in $\sigma$ is $a_k$. If $\sigma$
is $c_1 c_2 \ldots c_k a_1 a_2 \ldots a_k$, then we are done. Thus we
may assume that there is a $t$, $2 \le t \le k$, such that
$a_{t-1},c_{1},c_{2},\ldots,c_k,a_{t},a_{t+1},\ldots,a_k$ is a suffix
of $\sigma$.

In order to define the $n$--half-period $\Pi'$ claimed by the
proposition, we establish some facts regarding $\Pi$.

\begin{description}

\item{(A)} {\sl
Let $\pi_{i+1}$ be the permutation where $c_1$
becomes free. Then $\pi_{i}$ is of the form
$(a_k, a_{k-1}, \ldots ,a_{t},$ $d_{1},$ $d_{2},$
$\ldots,d_{p} c_1, c_2, \ldots c_k $) where
$p=t-1+m$ and each $d_{j}$ is either a $b$ or a free $a$.}

\end{description}

\noindent{\em Proof of } (A).
The perfectness of $\Pi$ readily
implies that every transposition in the $A$--Zone that involves an
element in $L:=\{a_t, a_{t+1}, \ldots, a_k\}$ is a double--discovery
transposition.  In particular, the first element that moves an element
in $L$ must involve a $c$. Therefore, as long as no $c$ becomes
free, all the elements in $L$ must stay in their original
position. Finally, we observe that when  $c_1$ becomes free,
$a_1, a_2, \ldots, a_{t-1}$ are already free, so each $d_j$ is either
a $b$ or a free $a$, as claimed.

\begin{description}

\item{(B)} {\sl No element in $\{a_k a_{k-1} \ldots a_t d_1, \ldots,
  d_{t-1}\}$ (these are the elements that are in the $A$--Zone, in the
  given order, in $\pi_i$) leaves the $A$--Zone before $c_k$ becomes
  free.}

\end{description}

\noindent{\em Proof of } (B).  Seeking a contradiction, let $e$ be the
  first element in $\{a_k a_{k-1} \ldots a_t d_1, \ldots, d_{t-1}\}$
  that moves out of the $A$--Zone before $c_k$ becomes free. The
  perfectness of $\Pi$ readily implies that the element that takes $e$
  out of the $A$--Zone is some $c_j$ (where by assumption $j \neq k$).
  Now right after $c_j$ swaps with $e$, $c_j$ and $c_k$ are in the
  $A$-- and $C$--Zones, respectively.  In particular, at this point
  $c_j$ and $c_k$ have not swapped. Now as we observed above, every
  transposition in the $A$--Zone involving an element in $L$ is
  double--discovery, and so it follows that $c_j$ never gets beyond
  (to the left of) the position $k-j+1$.  No matter where the $(c_j,c_k)
  \mapsto (c_k,c_j)$
  transposition occurs, this implies that $c_j$ must at some point be
  in a position $r$,  with $k-j+1 \le r \le k$, and then
move (right) to position $r+1$.
  Now in order to reach its final position, $c_j$ must eventually move
  back to the position $r$, via some transposition
  $\varepsilon = (x,c_j) \mapsto (c_j,x)$.  Since
  $\Pi$ is perfect, and $\varepsilon$ occurs in the $A$--Zone,
  $\varepsilon$ is a discovery transposition. But it clearly cannot be
  discovery for $c_j$, since $c_j$ is re-visiting the position $r$.  Now
  $x \in \{ a_k, a_{k-1}, \ldots, a_t, d_1, \ldots,d_{t-1} \}$, since
  these were the elements to the left of $c_j$ when it
  first entered the $A$--Zone.  Clearly $x$ cannot be a $d$, since
  each $d$ is either a $b$ or a free $a$, and $\varepsilon$ must
  be discovery for $x$. Thus $x$ must be in $L=\{a_k,a_{k-1},\ldots,
  a_t\}$.  But this is also impossible, since (see
  Proof of (A)) every transposition that involves an element in $L$
  must be a double--discovery transposition.

\begin{description}

\item{(C)} {\sl Suppose that two elements that are in the $A$--Zone
  (respectively, $C$--Zone)
  in $\pi_i$ transpose with each other in
  the $A$--Zone   (respectively, $C$--Zone) after $\pi_i$.  Then at least
  one of these elements leaves the $A$--Zone   (respectively, $C$--Zone) after $\pi_i$ and before
  this transposition occurs.}
\end{description}

\noindent{\em Proof of} (C).
First we note that the elements that are in the $C$--Zone in $\pi_i$
are $c_1, c_2, \ldots, c_k$, in this order, and that if two of them
transpose before at least one of them leaves the $C$--Zone, this
transposition would be confined, contradicting the assumption that
$\Pi$ is perfect.  That takes care of the $C$--Zone part of (C).

Now we recall that the elements that are in the $A$--Zone in $\pi_i$ are
$a_k, a_{k-1}, \ldots, a_t, d_1, d_2, \ldots, d_{t-1}$, in this order.
Suppose that two such elements transpose in the $A$--Zone after
$\pi_i$, and that between $\pi_i$ and this transposition (call it
$\lambda$) none of them
leaves the $A$--Zone.  It follows from the perfectness of $\Pi$ that,
for each $a_j$, every move of $a_j$ until it leaves the $A$--Zone
must involve some $c_\ell$. Thus none of the elements involved in
$\lambda$ can be an $a_j$, that is, both must be $d_j$'s.  But such a
transposition would clearly not be discovery (recall that each $d$ is
a free $a$ or a $b$), contradicting the
perfectness of $\Pi$.  This completes the proof of (C).

\begin{description}

\item{(D)} {\sl After $\pi_i$, the elements in the $A$--Zone leave it
  in the order $d_{t-1}, d_{t-2}, \ldots, d_1, a_t, \ldots, a_{k-1},
  a_k$, and the elements in the $C$--Zone leave it in the order $c_1,
  c_2, \ldots, c_k$.}

\end{description}

\noindent{\em Proof of} (D).  This is an immediate corollary of (C).

Now define $\Pi':=(\pi_i, \pi_{i+1} ,\ldots, \pi_{\binom{n}{2}}) =
(\pi_0^{-1}, \pi_1^{-1}, \ldots, \pi_{i-1}^{-1}, \pi_{i}^{-1})$.  It is
straightforward to check that $\Pi'$ is an $n$--half-period.  Define
the relabeling $a_i \mapsto a_i'$ for $i = t, t+1, \ldots, k$; $d_s
\mapsto a_{t-s}'$ for $s=1,\ldots, t-1$; $d_s \mapsto b_{s-t+1}'$ for $s =
t, t+1, \ldots, p$; and $c_i \mapsto c_i'$ for $i=1,\ldots, k$, so
that the initial permutation of $\Pi'$ (namely
$\pi_i = (a_k a_{k-1}
\ldots a_{t} d_{1} d_{2}$ $\ldots,d_{p} c_1 c_2 \ldots c_k)$)
 is $ (a_k' a_{k-1}' \ldots a_1'b_1' b_2' \ldots b_m' c_1' c_2' \ldots
c_k') $.

To complete the proof, we check that (i) the liberation sequence
of $\Pi'$ is $c_1' c_{2}' \ldots c_k' a_1' a_2' \ldots a_k'$; and that
(ii) $\Pi'$ is perfect.  We note that (i) follows immediately from (B)
and (D). Now
in view of Proposition~\ref{prop:perfect}, in order to prove that
$\Pi'$ is perfect it suffices to show that it has no confined
transpositions, and that it has exactly
$3\binom{k+1}{2}$ $(\le
k)$--critical transpositions. From (C) it follows that $\Pi'$ has no confined
transpositions. On the other hand, an application of
Proposition~\ref{prop:perfect}
to $\Pi$ (which is perfect) yields that $\Pi$ has $3\binom{k+1}{2}$
$(\le k)$--critical transpositions.  The construction of $\Pi'$
clearly reveals that $\Pi$ and $\Pi'$ have the same number of $(\le
k)$--critical transpositions, and so $\Pi'$ has $3\binom{k+1}{2}$
$(\le k)$--critical transpositions, as required.
\end{proof}

\noindent{\em Proof of Theorem~\ref{thm:main}}

Let $\Pi$ be an $n$--half-period of $\bf\Pi_P$, for $n$ a multiple of $3$. By the hypothesis  of the Main Theorem and the fact
$E_{\le k-1}(P)=N_{\leq k}(\Pi)$, we have $N_{\leq k}(\Pi)=3\binom{k+1}{2}$ for each $1\le k\le n/3$.
This equality and  Proposition~\ref{prop:noconf}  guarantee that $\bf\Pi_P$ contains an $n$--half-period, say $\Pi_P$, that satisfies  the hypothesis of Proposition~\ref{prop:perfect}. Thus  $\Pi_P$ is perfect, and   using Proposition~\ref{prop:3decomp} we get an $n$--half-period which behaves as we need for
$\bf\Pi_P$ to be 3--decomposable. \qed

\section{On Allowable Sequences That Minimize The Crossing-Number of \texorpdfstring{$K_{30}$}{K30}}\label{sec:k30}

This section is devoted to study of allowable sequences which come from configurations of $30$ points that minimize the crossing-number. 
In particular, each result presented in this section is focused on establish features of such sequences. Later, in Section~\ref{sec:proofk30}, each of these properties will be used in the proof of Theorem~\ref{thm:second}.

We begin by proving, with the help of Theorem~\ref{thm:main}, that all optimal sequence of $K_{30}$ are $3$-decomposable.

We have the following bounds given by \'Abrego et al.~\cite{Ab:hallin} for any $n$--half-period {\bf $\Pi$} of an allowable sequence.
\begin{equation}\label{eq:hallin}
N_{ \lfloor n/2 \rfloor } (\Pi) \leq
\begin{cases}
\lfloor \frac{1}{2}\binom{n}{2} - \frac{1}{2}N_{\leq \lfloor n/2 \rfloor -2}(\Pi) \rfloor, & \text{ if } n \text{ is even},\\
\lfloor \frac{2}{3}\binom{n}{2} - \frac{2}{3}N_{\leq \lfloor n/2 \rfloor -2}(\Pi) + \frac{1}{3} \rfloor, & \text{ if } n \text{ is odd}.
\end{cases}
\end{equation}
and
\begin{equation}\label{eq:nhallin}
N_{\leq \lfloor n/2 \rfloor -1} (\Pi) \geq
\begin{cases}
\binom{n}{2}- \lfloor \frac{1}{24}n(n+30) - 3 \rfloor, & \text{ if } n \text{ is even},\\
\binom{n}{2}- \lfloor \frac{1}{18}(n-3)(n+45) +\frac{1}{9} \rfloor, & \text{ if } n \text{ is odd}.
\end{cases}
\end{equation}

Now, if $\Pi$ is a $30$--half-period associated to a generalized configuration $P$ of $30$ points,  then from~(\ref{eq:hallin}) we know that $N_{15}(\Pi) \leq 72$ and if we combine~(\ref{eq:edg}) and~(\ref{eq:nhallin}) we get that $N_{14}(\Pi) \geq 72$. With this bounds in~(\ref{eq:credg}) we have $9723$ as a lower bound for $\pcr(K_{30})$. Moreover, if for some $k = 0, \ldots, 12$,~(\ref{eq:edg}) is not tight, then a simple calculation in (\ref{eq:credg}) shows that $\widetilde{cr}(P) \geq 9727$ and therefore $P$ will be worse than the best known configuration given implicitly by Aichholzer and Krasser in~\cite{AiKr}, which establishes $9726$ as an upper bound. Besides $72 \leq N_{14}(\Pi) \leq 75$ or $\pcr(P) \geq 9727$. So, in an optimal configuration with $30$ points,~(\ref{eq:edg}) must be tight for each $k = 0, \ldots, 12$ and so, by  the Main theorem, $P$ is $3$--decomposable.

For the remainder of this subsection, let us assume that $\Pi = (\pi_0, \pi_1, \ldots, \pi_{\binom{30}{2}})$ is a $3$--decomposable $30$--half-period, with initial permutation $\pi_0 = (a_{10}, a_9, \ldots, a_1, b_1, b_2, \ldots, b_{10}, c_1, c_2, \ldots, c_{10})$ and $A = \{a_{10}, a_9\-,\ldots,a_1\}$,  $B = \{b_1, b_2, \ldots, b_{10}\}$ and $C = \{c_1, c_2, \ldots, c_{10}\}$.

In order to count the number of $(\leq k)$--critical transposition in $\Pi$, we define two types of transpositions. A transposition is {\em monochromatic} if it occurs between two elements of the same set $A$, $B$ or $C$, otherwise is called {\em bichromatic}. We denote  the number of monochromatic (respectively, bichromatic) $(\leq k)$--critical transpositions in $\Pi$ by $N_{\leq k}^{mono}(\Pi)$ (respectively, $N_{\leq k}^{bi}(\Pi))$. Note that $N_{\leq k}(\Pi) = N_{\leq k}^{mono}(\Pi) + N_{\leq k}^{bi}(\Pi)$.

From~\cite{Ab:3sim} we get the next account for bichromatic transpositions on a $3$--decomposable $n$--half-period $\Pi'$:
\begin{equation}\label{eq:bicro}
N_{\leq k}^{bi}(\Pi') =
\begin{cases}
3\binom{k+1}{2} & \text{ if } k \leq n/3,\\
3\binom{n/3 + 1}{2} + (k - n/3)n & \text{ if } n/3 < k < n/2.
\end{cases}
\end{equation}

As a consequence of (\ref{eq:bicro}) we have the next two Corollaries:

\begin{cor}\label{cor:bi1}
$N_k^{bi}(\Pi) = 3k$ for $k = 1, 2, \ldots, 10$.
\end{cor}

\begin{cor}\label{cor:bi2}
$N_k^{bi}(\Pi) = 30$ for $k = 11, 12, 13, 14$.
\end{cor}

\begin{lem}\label{lem:bihal}
$N_{15}^{bi}(\Pi) = 15$.
\end{lem}
\begin{proof}
The number of bichromatic transpositions between $A$ and $B$ is $100$ because there is,  exactly, one bichromatic transposition for each element of $A\times B$. For the same reason there are $100$  bichromatic transpositions  between  $A$ and $C$ and $100$  between $B$ and $C$. So $N_{\leq 15}^{bi}(\Pi) = 300$. The desired result it follows from  Corollaries~\ref{cor:bi1} and~\ref{cor:bi2}  and the fact that $N_{15}^{bi}(\Pi)=300-\sum^{14}_{k=1}N_k^{bi}(\Pi)$.
\end{proof}

From the  above discussion, Corollary~\ref{cor:bi1}  and  Theorem~\ref{thm:main} it follows that all monochromatic transpositions occur in the {\em middle third}. Where the middle third is the space from the $11$th--position to $20$th--position.

\subsection{Digraphs}\label{subsec:dig}

Let $\Pi$ be a $3$--decomposable $n$--half-period of an allowable sequence ${\bf \Pi}$. A transposition between elements in the positions $i$ and $i + 1$ with $k < i < n-k$ is called a {\em $(> k)$--transposition}. All these transpositions are said to occur in the {\em $k$--center}. Let us denote the number of monochromatic transpositions that occur in the $k$--center and are of the kinds $aa$,  $bb$, and $cc$ by $N_{>k}^{aa}(\Pi)$, $N_{>k}^{bb}(\Pi)$, and $N_{>k}^{cc}(\Pi)$, respectively. Since each monochromatic transposition is an $aa$-- or $bb$-- or $cc$--transposition,  then $N_{>k}^{aa}(\Pi)+N_{>k}^{bb}(\Pi)+N_{>k}^{cc}(\Pi)$ is the total number of monochromatic transpositions that occur in the $k$--center.

Let $D_k$ be the digraph with vertex set $\{n/3, n/3 -1, \ldots,1\}$, and such that there is a directed edge from $i$ to $j$ if and only if $i > j$ and the transposition $a_ia_j$ occurs in the $k$--center.  Note that the number of edges of $D_k$ is exactly $N_{>k}^{aa}(\Pi)$.

In order to count the edges in $D_k$, let $\mc{D}_{v,m}$ be the class of all digraphs on $v$ vertices,
say $v,v-1,\ldots, 1$,  satisfying that $[i]^+ \leq m + [i]^-$ for all $v \geq i \geq 1$, where $[i]^+$ and $[i]^-$ denote the outdegree and the indegree of the vertex $i$, respectively, and if we have an edge from $i$ to $j$, $i \rightarrow j$, then $i > j$. Let $D_0(v,m)$ be the graph in $\mc{D}_{v,m}$ with vertices $v, v-1, \ldots,1$ recursively defined by
\begin{itemize}
\item $[v]^- = 0$,
\item $[i]^+ = \min \{ [i]^- + m , i-1 \}$ for each $v \geq i \geq 1$, and
\item for all $v \geq i >  j \geq 1$, $i \rightarrow j$ if and only if $i - 1 \geq j \geq i - 1 - [i]^-$.
\end{itemize}
Balogh and Salazar prove  in~\cite{BaSa}  that the maximum number of edges of a digraph in $\mc{D}_{v,m}$ is attained by $D_0(v,m)$. We note that $D_k$ is in $\mc{D}_{n/3,n-2k-1}$, and hence the number of edges in $D_k$ is bounded above by the number of edges in $D_0(n/3,n-2k-1)$.

\vspace{.5cm}

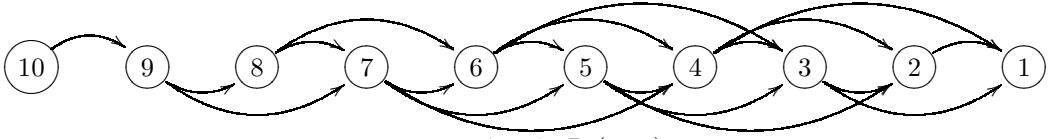
\begin{figure}[h]
\[ \entrymodifiers={++[o][F-]}
\xymatrix{
	10 \ar@/^{1pc}/[r]
	& 9 \ar@/_{0.8pc}/[r] \ar@/_{1.5pc}/[rr]
	& 8 \ar@/^{0.8pc}/[r] \ar@/^{1.5pc}/[rr]
	& 7 \ar@/_{0.8pc}/[r] \ar@/_{1.5pc}/[rr] \ar@/_{2pc}/[rrr]
	& 6 \ar@/^{0.8pc}/[r] \ar@/^{1.5pc}/[rr] \ar@/^{2pc}/[rrr]
	& 5 \ar@/_{0.8pc}/[r] \ar@/_{1.5pc}/[rr] \ar@/_{2pc}/[rrr]
	& 4 \ar@/^{0.8pc}/[r] \ar@/^{1.5pc}/[rr] \ar@/^{2pc}/[rrr]
	& 3 \ar@/_{0.8pc}/[r] \ar@/_{1.5pc}/[rr]
	& 2 \ar@/^{0.8pc}/[r]
	& 1
           } \]
\caption{\small Digraphs $D_0(10,1)$}
\label{fig:digra}
\end{figure}

From the preceding information, we can deduce that the number of edges in $D_{14}$ is at most $20$ (Figure~\ref{fig:digra}). This means that $N_{15}^{aa}(\Pi) \leq 20$, $N_{15}^{bb}(\Pi) \leq 20$, and $N_{15}^{cc}(\Pi) \leq 20$. Similarly, the number of edges in $D_{13}$ is at most $33$ and we know that $\binom{30}{2} - N_{\leq 13}(\Pi) = 144$ because all the bounds for $(\leq k)$--sets, for $k = 1, \ldots, 13$, are tight. Thus $N_{14}(\Pi) + h(\Pi) = 144$, besides from Corollary~\ref{cor:bi2} and Lemma~\ref{lem:bihal} we get that $N_{14}^{bi}(\Pi) + N_{15}^{bi}(\Pi) = 45$.  This implies that $N_{>13}^{mono}(\Pi) = 99$ and therefore there are exactly $33$--monochromatic transpositions in the $13$--center per each set $A$, $B$ and $C$.

\begin{lem}\label{lem:dig33}
If $D$ is a digraph in $\mc{D}_{10,3}$ with $33$ edges,  then for $i,j = 10, 9, 8, 7$ and $i > j$ there is an edge  from $i$ to $j$.
\end{lem}
\begin{proof}
Clearly, the number of edges  with tail in $\{10, 9, 8, 7\}$ and head in $\{6, 5, \ldots, 1\}$ is at most $12$ and the number of edges in the vertex set $\{6, 5, \ldots, 1\}$ is at most $15$ (this is attained by $D_0(6,3)$). Then we need the $6$ edges between the elements in  $\{10, 9, 8, 7\}$ in order to get the $33$ edges in $D$.
\end{proof}

\subsection{Restrictions in the monochromatic transpositions}

From now on, we shall use $\Pi = (\pi_0, \pi_1, \ldots, \pi_{\binom{30}{2}})$ to denote a
$3$--decomposable $30$--half-period of an optimal configuration for $K_{30}$ and
$\pi_0 = (a_{10}, a_9, \ldots, a_1,b_{l_1}, \ldots, b_{l_{10}}, c_1, c_2, \ldots, c_{10})$ to denote its first permutation. Also we assume that $A:=\{a_{1}, a_2,\ldots, a_{10}\}$, $B:=\{b_{l_1}, b_{l_2}, \ldots, b_{l_{10}}\}$ and $C:=\{c_{1}, c_2 \ldots, c_{10}\}$.

 As $\Pi$ is $3$--decomposable and all monochromatic transpositions occur in the middle third, it follows that there is a unique element of $B$ that reaches the position $1$ (or $30$). We shall denote by $b_{10}$ to such element of $B$. For the same reasons,  for $i=2,3,\ldots ,10$,  there is a unique element of $B$, which we denote by $b_{10-i+1}$,  that reaches the position $i$ (or $30-i+1$) but not the position $i-1$ (or $30-i+2$). Clearly, $B=\{b_{1}, b_{2}, \ldots, b_{10}\}$.

In this subsection we use that in $\Pi$ the lower bound given in  (\ref{eq:edg}) is tight for
$k =0,\ldots ,12$ in order to deduce some restrictions about  the monochromatic transpositions.

\begin{rem}\label{remark1}
Because $\Pi$ is $3$--decomposable  ($A$ can interchange the role with $B$ or $C$), everything that we say for $A$ is also valid for $B$ or $C$.
\end{rem}

\begin{lem}\label{lem:11-13}
Each transposition of $\Pi$ that contributes to $N^{mono}_{11}(\Pi)+N^{mono}_{12}(\Pi)+N^{mono}_{13}(\Pi)$ involves to some of $a_{10}, a_9,  a_8, b_{10}, b_9, b_8, c_{10}, c_9$ or $c_8$
\end{lem}

\begin{proof} Since we have exactly $33$ monochromatic transpositions in the $13$--center, then,  by Lemma~\ref{lem:dig33}, mandatory the transpositions between elements of $\{a_{10}, a_9, a_8, a_7\}$ occur in the $13$--center.

From the equation~(\ref{eq:bicro}) and the fact that  (\ref{eq:edg}) is tight for $k =0,\ldots ,12$, we get that $N_{11}^{mono}(\Pi) = 6$, $N_{12}^{mono}(\Pi) = 12$ and $N_{13}^{mono}(\Pi) = 18$. Because no other $a$ is behind $a_{10}$, it is not possible to have more than one monochromatic transposition per gate involving $a_{10}$.  Furthermore, $a_{10}$ should change with $a_9, a_8,\ldots, a_1$ in the $10$--center (middle third). Thus $a_{10}$ has one monochromatic transposition in each gate of the middle third. By Remark~\ref{remark1} the same happen with $b_{10}$ and $c_{10}$. Thus,  the $2 \cdot 3 $ monochromatic transpositions due to $a_{10}$, $b_{10}$ and $c_{10}$ are all the monochromatic transpositions associated with $N^{mono}_{11}(\Pi)$.

For the preceding, every monochromatic transposition involving $a_9$ occurs in $11$--center. Since the swap between $a_{10}$ and $a_9$ occurs in the $13$--center, thus $a_9$ contributes $2$  to $N^{mono}_{12}(\Pi)$. So we have $4$ different monochromatic transpositions due to $a_{10}$ and $a_9$. By Remark~\ref{remark1}, we get $2 \cdot 3 + 2 \cdot 3$ monochromatic transpositions due to $a_{10}, a_9, b_{10}, b_9$ and $c_{10}, c_9$ and they are all the monochromatic transpositions associated with $N^{mono}_{12}(\Pi)$.

So each  monochromatic transposition involving $a_8$ occurs in the $12$--center.  Thus $a_8$ contributes $2$  to $N^{mono}_{13}(\Pi)$. $a_{10}$ and $a_9$ also have others $2$ monochromatic transpositions there, and all the transpositions are different because $a_{10}, a_9$ and $a_8$ change in the $13$--center. Hence the $2 \cdot 3 + 2 \cdot 3 + 2 \cdot 3$ monochromatic transposition associated with $N^{mono}_{13}(\Pi)$ are generated by $a_{10}, a_9,  a_8, b_{10}, b_9, b_8, c_{10}, c_9$ and $c_8$.
\end{proof}

Let $k\in \{10,11,\ldots ,14\}$. Note that every element in a $3$--decomposable $30$--half-period $\Pi'$ occupies each position of the $10$--center at least once. From now on, if $\tau$ is the first  (respectively, last) transposition in which $x\in A\cup C$ enters (respectively, leaves) the $k$--center, then we say that $\tau$ is the swap in which $x$ {\em enters} (respectively,  {\em leaves})  the $k$--center of $\Pi'$.

\begin{lem}\label{remark2}

For $x\in \{a, c\}$, the elements $x_1, x_2,\ldots, x_{10}$ enter (respectively, leave) the $13$--center of $\Pi$ in ascending (respectively, descending) order. Moreover, for $i=1,2,\ldots,7$;

\begin{description}
\item{(1)}   the swap between $a_i$ and $b_{7-i+1}$ occurs in the $13$th--gate and it is precisely the swap in which $a_i$ enters (and $b_{7-i+1}$ {\em leaves}) the $13$--center of $\Pi$,

\item{(2)}  the swap between $a_{7-i+1}$ and $c_i$ occurs in the $17$th--gate and it is precisely the swap in which $a_{7-i+1}$ leaves (and $c_i$  enters) the $13$--center of $\Pi$ and,

\item{(3)} the swap between $b_i$ and $c_{7-i+1}$ occurs in the $13$th--gate and it is precisely the swap in which $c_{7-i+1}$ leaves (and $b_i$ {\em enters}) the $13$--center of $\Pi$. \end{description}
It follows from (3) (respectively, (1)) that $b_1, b_2,\ldots, b_{7}$  also enter (respectively, leave) the $13$--center of $\Pi$ in ascending (respectively, descending) order.

\end{lem}

\begin{proof}  By Lemma~\ref{lem:dig33} and the fact that there are exactly $33$ monochromatic transpositions in the $13$--center of $\Pi$, each transposition between elements of $\{x_{10}, x_9, x_8, x_7\}$ occurs in the $13$--center. Also, by Lemma~\ref{lem:11-13}, each transposition between elements of $\{x_{7}, x_6,\ldots, x_1\}$ occurs in the $13$--center. Together, these two conclusions, imply that the elements of  $\{x_{10}, x_9,\ldots, x_1\}$ enter (respectively, leave) the $13$--center of $\Pi$ in ascending (respectively, descending) order. 

We only show (1). The parts (2) and (3) are analogous.
 
Let $w\in \{a, b\}$. Because all monochromatic transpositions of $\Pi$ occur in the 
$10$--center, then the elements of $\{w_{10}, w_9,\ldots, w_1\}$ enter (respectively, leave) the $10$--center of $\Pi$ in ascending (respectively, descending) order. 

For $t=1,2,3$ we know (Lemma~\ref{lem:11-13}) that every  monochromatic transposition involving $b_{10-t+1}$ occurs in the $(10+t-1)$--center. This and the fact that the $b$'s leave the $10$--center in descending order  imply that the swap between $a_1$ and $b_{10-t+1}$ occurs in the $(10+t-1)$th--gate. 

Since  (Lemma~\ref{lem:11-13}) each transposition between elements of $\{b_{7}, b_6,\ldots, b_1\}$ occurs in the $13$--center and they leave the $10$--center in descending order, then the swap where $a_j$ enters in the $13$--center must be with $b_{7-j+1}$, where $j=1,2,\ldots,7$.
\end{proof}

\begin{lem}\label{lem:restrcent}
Let $\pi_{a_{10}} $ be the permutation of $\Pi$ where $a_{10}$ enters in the $13$--center. Then $\pi_{a_{10}} $ looks like $$(B, a_{\leq 4}, a_{\leq 5}, a_{\leq 6}, a_{10}, a_i, a_j, a_k, a_{\leq 6}, a_{\leq 5}, a_{\leq 4}, C)$$ where $a_{\leq p}$ is an $a_u$ with $1 \leq u \leq p$, further $\{i,j,k\}=\{7, 8, 9\}$.
\end{lem}
\begin{proof} For $j=7,6,\ldots, 1$ let $\tau_j$ be the transposition  in which $a_j$ enters in the $13$--center. So,  when $\tau_5$ occurs  there is at least one $r \in \{1, 2, 3, 4 \}$ such that $a_r$ is to the right hand side of the $13$--center (without loss of generality, we assume that $a_r$ is the  rightmost $a$ element). By Lemma~\ref{lem:11-13},  all the monochromatic transpositions between elements of
$\{a_7,a_6, \ldots ,a_1\}$ or between elements of $\{a_{10},a_9,a_8, a_7\}$ occur in the $13$--center. Thus $a_r$  does not move to the left until after $a_{10}$ exits of the $13$--center. On the other hand, since all monochromatic transpositions occur in the middle third, thus, when $a_{10}$ enters in the $13$--center $a_r$ must be at position $20$. Using  similar arguments  with $\tau_6$ and $\tau_7$ we get the restriction on the right hand side.

Let $a_{l_j}$ be the $a$  that swaps with  $a_{10}$ in the $(14-j)$th--gate (where $j=3,2,1$). Since each $aa$ transposition that contributes to $N_{11}(\Pi)+N_{12}(\Pi)+N_{13}(\Pi)$ involves to $a_{10}$, $a_9$ or $a_8$ and the transpositions between elements of $\{a_{10}, a_9, a_8, a_7\}$ occur in the $13$--center,  then  $l_j\leq 6$. Thus $a_{l_j}$ needs $j$  transpositions of kind $a_{l_j}c$ in order to move  to $13$--center.  Hence  $a_{l_j}$ will remain to the left hand side of the $13$--center until after $c_j$ enters in the $13$--center. But, by Lemma~\ref{remark2}, when $c_j$ enters in the $13$--center all $a_n$'s with $n \geq 8-j$ have left from there. Hence $l_j \leq 7-j$.
\end{proof}

Let $\hal(a_j)$ denote the number of $a_i$ elements, $i < j$, such that $a_j$ changes with $a_i$ in the $15$th--gate. This means, the outdegree of the vertex $a_j$ in the digraph $D_{14}$ associated to $N_{> 14}^{aa}(\Pi)$.

Some facts are easier to see in $\Pi^*$, the {\em reverse half-period of}  $\Pi$. We define the reverse half-period of $\Pi$ as $\Pi^* = (\pi_0^*, \pi_1^*, \ldots, \pi_l^*, \ldots, \pi_{\binom{30}{2}}^*) := (\pi_{\binom{30}{2}}^{-1}, \pi_{\binom{30}{2}-1}^{-1}, \ldots, \pi_{\binom{30}{2} - l}^{-1}, \ldots, \pi_0^{-1})$. It is clear that $\Pi$ and $\Pi^*$ have the same combinatorial properties.

\begin{lem}\label{lem:hal(a_i)}
Let $\pi_{a_{10}} $ be the permutation of $\Pi$ where $a_{10}$ enters in the $13$--center. If $a_i$, $1 \leq i \leq 5$, is at position $10 + l$ or at position $20-l+1$, $1 \leq l \leq 3$, then $\hal(a_i) \leq l $
\end{lem}
\begin{proof}
We just prove the case when $a_i$ is at position $10 + l$, otherwise we look at $\Pi^*$. Let $B(a_i)$  be the set of $l-1$ $a$'s that are behind of $a_i$ in $\pi_{a_{10}} $. Let $j$ be the number of element in $B(a_i)$ with  index  smaller than $i$. This means that in $\pi_{a_{10}} $,  $a_i$ has already changed with each element of $B(a_i)$ with  index  smaller than $i$. Note that  these transpositions contribute at most $j$ to $\hal(a_i)$.  On the other hand,   each element of $B(a_i)$ with   index greater than $i$ moves $a_i$ to the left one time, then $a_i$ could make at most $((l-1)-j)+1$  transpositions in the $15$th--gate which involve an $a$  with index smaller  than $i$. Thus $\hal(a_i)\le j+(((l-1)-j)+1)=l$.
 \end{proof}

\begin{cor}\label{cor:no20}
$N_{15}^{aa}(\Pi) \leq 19$, $N_{15}^{bb}(\Pi) \leq 19$ and $N_{15}^{cc}(\Pi) \leq 19$
\end{cor}
\begin{proof}
What we say for $A$ also apply for $B$ and $C$.  By Lemmas~\ref{lem:restrcent} and~\ref{lem:hal(a_i)}, $\hal(a_4) + \hal(a_5) \leq 5$ and hence the digraph $D_{14}$ associated to $N_{> 14}^{aa}(\Pi)$ has at most $19$ edges: at most $5$ edges with tail in $\{a_{10}, a_9, a_8, a_7, a_6 \}$ and head in $\{ a_5, a_4, a_3, a_2, a_1 \}$, at most $6$ edges between the elements of $\{a_{10}, a_9, a_8, a_7, a_6 \}$, at most $5$ edges with tail in $\{ a_5, a_4 \}$, and at most $3$ edges between the elements of $\{a_3, a_2, a_1\}$.
\end{proof}

\begin{rem}\label{rem:19hal}
In fact, if we want to have $19$ halvings, thus $\hal(a_{10}) + \hal(a_9) + \cdots + \hal(a_6)$ must be $11$, $\hal(a_5) + \hal(a_4)$ must be $5$ and $\hal(a_3) + \hal(a_2) + \hal(a_1)$ must be $3$. The later means that $a_3, a_2, a_1$ have to change in the $15$th--gate.
\end{rem}

\begin{cor}\label{cor:2orilla}
If $N_{15}^{aa}(\Pi) = 19$, then in the permutation $\pi_{a_{10}} $ of $\Pi$ in which $a_{10}$ enters in the $13$--center, $a_1$ and $a_2$ are at positions $11$ and $20$, respectively, or vice versa.
\end{cor}
\begin{proof}
From  Lemma~\ref{lem:hal(a_i)} and Remark~\ref{rem:19hal} it follows that $a_4$ is not at position $11$ or $20$ in
$\pi_{a_{10}} $. On the other hand,  by Lemma~\ref{lem:restrcent} we know that  $a_6$ is at position $13$
(position $18$), then $a_4, a_5$ occupy the positions $18$ and $19$ (positions $12$ and $13$) or they occupy the positions $12$ and $18$ (positions $13$ and $19$), not necessarily in that order. Because $\hal(a_3)$ must be $2$, then, by the Lemma~\ref{lem:hal(a_i)} and with the prior discussion, $a_3$ must be at position $12$  or $19$. So we get that $a_1, a_2$ are at positions $11$  and $20$, not necessarily in that order.
\end{proof}

Before proceeding with the proof of Theorem~\ref{thm:second}, we need to establish two more lemmas.

\begin{lem}\label{lem:5mid}
Let $\pi_{a_{10}} , \pi_{c_{10}} $ and $\pi_{b_{10}} $ be the permutations of $\Pi$ where $a_{10}, c_{10}$ and $b_{10}$ enter in the $13$--center, respectively. If $a_5$ is at position $12$ or $19$ in $\pi_{a_{10}} $, then $N_{15}^{aa}(\Pi) < 19$, $N_{15}^{bb}(\Pi) < 19$ and $N_{15}^{cc}(\Pi) < 19$.
\end{lem}

\begin{proof}
Suppose that  $a_5$ is at position $12$ in $\pi_{a_{10}} $ (the case when $a_5$ is at position $19$ is the same if we see $\Pi^*$). So $\pi_{a_{10}} $ looks like
\begin{equation}\label{p-a}
\pi_{a_{10}} =(B, a_{i_1}, a_5, a_{i_2}| a_{10} - - - | a_{i_3}, a_{i_4}, a_{i_5}, C).
\end{equation}

Since there are no $aa$--transpositions after of $\pi_{a_{10}} $ on the left hand side of the $13$--center,  $a_5$ moves  to the $13$--center by means of two $ac$--transpositions.  By Lemma~\ref{remark2}, the swap between $a_5$ and $c_3$ occurs in the $17$th--gate, and hence, $a_5$ is moved from the position $12$ to $13$--center by $c_1$ and $c_2$. On the other hand, because all the transpositions between elements of $\{c_1, c_2,\ldots ,c_7\}$ or between elements of  $\{c_7, c_8, c_9, c_{10}\}$  occur in the $13$--center, then when $c_{10}$ enters in the $13$--center, $c_1$ and $c_2$ are at positions $11$ and $12$, not necessarily in that order. So $\pi_{c_{10}} $ looks like
\begin{equation}\label{p-c}
\pi_{c_{10}} =(B, c_{1 \mbox{ \tiny or } 2}, c_{2 \mbox{ \tiny or } 1}, c_{j_1} | - - - c_{10} | c_{j_2}, c_{j_3}, c_{j_4}, A),
\end{equation}
and by Lemma~\ref{lem:restrcent}, $j_4\in \{3, 4\}$.

Now we deduce some restrictions on $\pi_{b_{10}} $. As before, since there are no $cc$--transpositions after of $\pi_{c_{10}}$ on the right hand side of the $13$--center,  $c_{j_4}$ moves  to the $13$--center by means of three $bc$--transpositions. By Lemma~\ref{remark2}, the swap between $c_{j_4}$ and $b_{7-j_4+1}$ occurs in the $13$th--gate, and hence, $c_{j_4}$ is moved from the position $20$ to $13$--center by three $b$'s, say  $b_{k_1}, b_{k_2}, $ and $b_{k_3}$, such that $k_1, k_2, k_3 < 7-j_4+1 \le 5$. Thus, when $\pi_{b_{10}} $ occurs, $b_{k_1}, b_{k_2}$ and $b_{k_3}$ are at positions $18, 19$ and $20$. So  $\pi_{b_{10}} $ looks like $(C, b_{k_6}, b_{k_5}, b_{k_4} | b_{10} - - - | b_{k_3}  , b_{k_2}, b_{k_1}, A)$.  Thus, by Lemma~\ref{lem:restrcent}, $k_4 = 6$ and $k_5 = 5$ and
$\pi_{b_{10}} $ looks like
\begin{equation}\label{p-b}
\pi_{b_{10}} =(C, b_{k_6}, b_5, b_6|b_{10} - - - | b_{k_3}, b_{k_2}, b_{k_1}, A).
\end{equation}

In a similar way that (\ref{p-c}) was obtained from (\ref{p-a}), it is possible to obtain (\ref{p+a})  (respectively, (\ref{p+b})) from (\ref{p-b})   (respectively, (\ref{p+c})); (\ref{p+c})  can be obtained
 from (\ref{p+a}) like (\ref{p-b}) was obtained from (\ref{p-c}).
\begin{equation}\label{p+a}
\pi_{a_{10} + \binom{30}{2}} = (C, a_{1 \mbox{ \tiny or } 2}, a_{2 \mbox{ \tiny or } 1}, a_{i_3}| - - - a_{10} | a_{i_2}, a_5, a_{i_1}, B).
\end{equation}
\begin{equation}\label{p+c}
\pi_{c_{10} + \binom{30}{2}} = (A, c_{j_4}, c_5, c_6| c_{10} - - - | c_{j_1}, c_{2 \mbox{ \tiny or } 1}, c_{1 \mbox{ \tiny or } 2}, B).
\end{equation}
\begin{equation}\label{p+b}
\pi_{b_{10} + \binom{30}{2}} = (A, b_{1 \mbox{ \tiny or } 2}, b_{2 \mbox{ \tiny or } 1}, b_{p}| - - - b_{10} | b_6, b_5, b_{k_6}, C).
\end{equation}
The desired result is immediate from (\ref{p+a}), (\ref{p+c}), (\ref{p+b}) and Corollary~\ref{cor:2orilla}.
\end{proof}

\begin{lem}\label{lem:1-2}
Let $\pi_{a_{10}} , \pi_{c_{10}} $ and $\pi_{b_{10}} $ as in
Lemma~\ref{lem:5mid}. If  $N_{15}^{aa}(\Pi) =19$ and for $x=a,b,c$; $x_j$ occupies the $11$th-- or $20$th--position in  $\pi_{x_{10}}$,  then $j\in \{1,2\}$.
\end{lem}

\begin{proof} We only prove  the case $x=c$ (the cases $x=a$ and $x=b$ are analogous). Suppose that $c_j$ occupies the $11$th-- or $20$th--position in $\pi_{c_{10}} $.

\noindent{\sc{Case 1}}. $c_j$ occupies the $11$th--position in $\pi_{c_{10}} $. Suppose that $a_t$ occupies the $13$th--position in $\pi_{a_{10}} $. By Lemma~\ref{lem:5mid} we know that  $t\in \{5,6\}$.

By Lemma~\ref{remark2}, the swap between $a_{t}$ and $c_{7-t+1}$ occurs in the $17$th--gate, and hence,  $a_{t}$ is moved from the position $13$ to $13$--center by a $c_r$ such that $r\leq7-t\leq 2$. On the other hand, by Lemma~\ref{lem:11-13} we know that $c_r$ does not have monochromatic transpositions on the left hand side of the $13$--center until after $\pi_{c_{10}} $ occurs. Thus $c_r=c_j$.

\noindent{\sc{Case 2}}. $c_j$ occupies the $20$th--position in $\pi_{c_{10}} $. Seeking a contradiction, suppose that  $j\notin \{1,2\}$. So by Lemma~\ref{lem:restrcent},  $j\in \{3,4\}$.
Again,  by Lemma~\ref{remark2}, the swap between $c_{j}$ and $b_{7-j+1}$ occurs in the $13$th--gate, and hence, $c_{j}$ is moved from the position $20$ to $13$--center by three $b$'s, say  $b_{j_1}, b_{j_2}, $ and $b_{j_3}$, such that $j_1, j_2, j_3< 7-j+1\le 5$. It follows from Lemma~\ref{lem:11-13} that none of $b_{j_1}, b_{j_2}, $ and $b_{j_3}$ moves until after $\pi_{b_{10}} $ occurs. This implies that $b_{j_1}, b_{j_2}, $ and $b_{j_3}$ occupy the positions $18,19$ and $20$ in $\pi_{b_{10}} $. By Lemma~\ref{lem:restrcent}, $b_5$ is in the $12$th--position and by Remark~\ref{remark1} and Lemma~\ref{lem:5mid}, $N^{aa}_{15}(\Pi) < 19$.
\end{proof}

\section{The Rectilinear Crossing Number of \texorpdfstring{$K_{30}$}{K30}: Proof of Theorem~\ref{thm:second}}\label{sec:proofk30}

Let $\pi_{a_{10}} $, $\pi_{b_{10}} $ and $\pi_{c_{10}} $ as in
Lemma~\ref{lem:5mid}. By Lemmas~\ref{lem:5mid} and~\ref{lem:1-2}, if $N_{15}^{aa}(\Pi) = 19$ then, without loss of generality, $\pi_{a_{10}} $looks like
\begin{equation}\label{2-a}
\pi_{a_{10}} =(B, a_{i_1}, a_{i_2}, a_6| a_{10} - - - | a_5, a_{i_3}, a_{i_4}, C),
\end{equation}
with $\{i_1,i_4\}=\{1,2\}$, otherwise we look $\Pi^*$, besides in the $13$--center are $ a_9, a_8, a_7$ in some order.

By Lemma~\ref{remark2}, $a_6$ leaves the $13$--center with $c_2$, so $a_6$ re-enters in the $13$--center with the transposition with $c_1$. Thus $c_1$ occupies the  $11$th--position  of  $\pi_{c_{10}} $. So  by  Lemma~\ref{lem:1-2}, $\pi_{c_{10}} $ looks like
\begin{equation}\label{2-c}
\pi_{c_{10}} =(B, c_1, c_{j_1}, c_{j_2} | - - - c_{10} | c_{j_3}, c_{j_4}, c_2, A).
\end{equation}
Again, since $b_6$ enters in the $13$--center with the swap with $c_2$, $\pi_{b_{10}} $ looks like $(C, b_{k_1}, b_{k_2}, b_{k_3} |b_{10} - - - | b_{k_4}, b_{k_5}, b_{k_6}, A)$ with $k_4, k_5, k_6 \leq 5$. Thus, by Lemmas~\ref{lem:restrcent} and~\ref{lem:5mid}, $\pi_{b_{10}} $ looks like
\begin{equation}\label{2-b}
\pi_{b_{10}} =(C, b_{k_1}, b_{k_2}, b_6 |b_{10} - - - | b_5, b_{k_5}, b_{k_6}, A).
\end{equation}

In a similar way that~(\ref{2-c}) was obtained from~(\ref{2-a}), it is possible to obtain~(\ref{2+a})  (respectively,~(\ref{2+b})) from~(\ref{2-b})   (respectively,~(\ref{2+c}));~(\ref{2+c})  can be obtained
 from~(\ref{2+a}) like~(\ref{2-b}) was obtained from~(\ref{2-c}).

\begin{equation}\label{2+a}
\pi_{a_{10} + \binom{30}{2}} = (C, a_1, a_{i_3}, a_5 | - - - a_{10} | a_6, a_{i_2}, a_2, B).
\end{equation}
\begin{equation}\label{2+c}
\pi_{c_{10} + \binom{30}{2}} = (A, c_2, c_{j_4}, c_6 | c_{10}- - - | c_5, c_{j_1}, c_1, B).
\end{equation}
\begin{equation}\label{2+b}
\pi_{b_{10} + \binom{30}{2}} = (A, b_1, b_{k_5}, b_5 | - - - b_{10}| b_6, b_{k_2}, b_2, C).
\end{equation}

So we have only two cases, when $i_2$ equals to $3$ or $4$.

\noindent{\sc{Case } $i_2 = 4$}.
The permutation $\pi_{a_{10}}$ is $(B, a_2, a_4, a_6 | a_{10} - - - | a_5, a_3, a_1, C)$. By Lemma~\ref{remark2}, $a_4$ leaves the $13$--center with $c_4$, then $a_4$ must re-enters to the $13$--center with $c_3$ and therefore $\pi_{c_{10}}$ is $(B, c_1, c_3, c_5 | - - - c_{10} | c_6, c_4, c_2, A)$, and for similar reasons, the permutation $\pi_{b_{10}}$ is $(C, b_2, b_4, b_6 | b_{10} - - - |b_5, b_3, b_1, A)$.

\begin{claim}\label{claim:246}
If $\hal(a_3) + \hal(a_2) + \hal(a_1) = 3$, then $\hal(c_5) \leq 2$. Hence $N^{cc}_{15}(\Pi) \leq 18$
\end{claim}

\noindent{\em Proof of Claim~\ref{claim:246}}.
 Since $N^{aa}_{15}(\Pi) = 19$, by Remark~\ref{rem:19hal}, $\hal(a_3) + \hal(a_2) + \hal(a_1) = 3$. By Lemma~\ref{remark2}, $a_3$ leaves the $13$--center swapping with $c_5$, and the permutation is
 $$(B,c_1, \{c_3, a_2\} |\{c_2, c_4, a_1 \} c_5 | a_3, ...),$$
 where the notation $\{ \ \}$ means that $c_2, c_4, a_1$ occupy those positions, but not necessarily in that order, similarly for $a_2$ and  $c_3$. Because $a_2$ must to change with $a_1$ in the $15$th--gate, this is only possible if $c_5$ changes with $a_1$ in the $15$th--gate, but then $c_5$ does not change with neither $c_2$ or $c_4$ in the $15$th--gate, and therefore $\hal(c_5) \leq 2$. $N^{cc}_{15}(\Pi) \leq 18$ is a consequence of the Remark~\ref{rem:19hal}. This completes the proof of Claim~\ref{claim:246}.

 If $N^{cc}_{15}(\Pi) = 18$ and with the fact that $\hal(c_5) \leq 2$,  by Remark~\ref{rem:19hal}, we conclude that $\hal(c_3) + \hal(c_2) + \hal(c_1) = 3$. Since $\pi_{c_{10}}$ has the same configuration as $\pi_{a_{10}}$, named $(B, c_1, c_3, c_5 | - - - c_{10}| c_6, c_4, c_2, A)$ and also satisfies the hypotheses of Claim~\ref{claim:246}, we conclude that $N^{bb}_{15}(\Pi) \leq 18$. Now if $N^{bb}_{15}(\Pi) = 18$, $B$ satisfies the Claim~\ref{claim:246} too  and implies that $N^{aa}_{15}(\Pi) \leq 18$, which is a contradiction. Then $N^{aa}_{15}(\Pi) = 19, N^{cc}_{15}(\Pi) = 18$ and $N^{bb}_{15}(\Pi) \leq 17$.

 So we suppose that $N^{cc}_{15}(\Pi) \leq 17$. The only case we have to worry about is when $N^{bb}_{15}(\Pi) = 19$, but recall that when $b_{10}$ enters in the $13$--center, the permutation $\pi_{b_{10}}$ is
 $$\pi_{b_{10}} = (C, b_2, b_4, b_6 |b_{10} - - - | b_5, b_3, b_2, A)$$
 and $B$ holds the hypotheses of Claim~\ref{claim:246}, which implies that $N^{aa}_{15}(\Pi) \leq 18$, and this is a contradiction. Thus $N^{aa}_{15}(\Pi) = 19, N^{cc}_{15}(\Pi) \leq 17$ and $N^{bb}_{15}(\Pi) \leq 18$. 

\vspace*{12pt}

\noindent{\sc{Case } $i_2 = 3$}.
So,  $\pi_{a_{10}}=(B, a_2, a_3, a_6 | a_{10} - - - | a_5, a_4, a_1, C)$.
By Lemma~\ref{remark2}, $a_3$ leaves the $13$--center with $c_5$, then $a_3$ re-enters to $13$--center with $c_3$ or $c_4$.

Suppose that $a_3$ re-enter with $c_3$, then $\pi_{c_{10}}$ looks like
$$\pi_{c_{10}} = (B, c_1, c_3, c_5 | - - - c_{10}| c_6, c_4, c_2, A),$$
but $c_4$ leaves the $13$--center with $b_4$, then $c_4$ must re-enter with $b_3$, so we have
$$\pi_{b_{10}} = (C, b_2, b_4, b_6 | b_{10}- - -| b_5, b_3, b_1, A),$$
but again, $b_4$ leaves the $13$--center with $a_4$, so $b_4$ re-enters with $a_3$, and then we get
$$\pi_{a_{10} + \binom{30}{2}} = (C, a_1, a_3, a_5 | - - - a_{10}| a_6, a_4, a_2, B),$$
which is a contradiction. Thus $a_3$ re-enters to the $13$--center with $c_4$.

Here, just by convenience we work in $\Pi^*$. Let $\pi^*_{a_{10}}$ be the permutation of $\Pi^*$ where $a_{10}$ enters in the $13$--center. So, $$\pi^*_{a_{10}} = (C,a_1, a_4, a_5|a_{10} - - -| a_6, a_3, a_2, B).$$

\begin{claim}\label{claim:145}
$b_2$ does not change with $b_1$ or, if $\hal(a_5) = 3$ then $b_3$ does not change with $b_1$ in the $15$th--gate. Moreover, in both cases $N^{bb}_{15}(\Pi) \leq 18$.
\end{claim}

\noindent{\em Proof of Claim~\ref{claim:145}}.
If $b_2$ does not change with $b_1$ in the $15$th--gate, by Remark~\ref{rem:19hal}, $N^{bb}_{15}(\Pi) \leq 18$.

So we assume that $b_2$ changes with $b_1$ in the $15$th--gate. Like $N^{aa}_{15}(\Pi) = 19$, by Remark~\ref{rem:19hal}  and Lemma~\ref{lem:hal(a_i)}, $\hal(a_5)$ is $3$. When $a_6$ leaves the $13$--center, this swap is with $b_2$, so in that moment we have the following situation
$$(...|\{a_2, a_3, b_1\} b_2| a_6,...).$$
When  $b_2$ changes with $b_1$ in the $15$th--gate, we have the following
$$(...|a_{2 \mbox{ \tiny or } 3}, b_2, b_1, a_{3 \mbox{ \tiny or } 2}|...),$$
$a_5$ re-enters in the $13$--center with $b_2$ and must to change with either $a_2$ or $a_3$ in the $15$th--gate to complete $3$ halvings because at most $a_5$ has changed in the $15$th--gate with $a_1$ and $a_4$, this implies that must be an $a$ in the $16$th--position and that is only possible if $b_1$ swaps with the leftmost $a$ of the $13$--center, and so when $a_5$ leaves the $13$--center and $b_3$ enters in it, the permutation is
$$(...|b_1, \{ a_{2 \mbox{ \tiny or } 3}, a_{3 \mbox{ \tiny or} 2} \}, b_3|a_5...),$$
but $a_4$ re-enters in the $13$--center with $b_3$, and there are no more $b$'s in the $13$--center until after $a_4$ leaves it, thus no one moves $b_1$ from the $13$th--position and therefore $b_3$ does not change with $b_1$ in the $15$th--gate.
This and Remark~\ref{rem:19hal} imply $N^{bb}_{15}(\Pi) \leq 18$. This completes the proof of Claim~\ref{claim:145}.

If $N^{bb}_{15}(\Pi)$ is $18$ and knowing that $\hal(b_3) + \hal (b_2) + \hal(b_1) \leq 2$, by Remark~\ref{rem:19hal} we get that $\hal(b_5)$ is $3$ and also we have the same configuration $(C,b_2, b_3, b_6 | - - - b_{10} | b_5, b_4, b_1, A)$. Then the hypotheses of the Claim~\ref{claim:145} are satisfied and consequently $N^{cc}_{15}(\Pi) \leq 18$.

But again, if $N^{cc}_{15}(\Pi) = 18$ and $\hal(c_3) + \hal(c_2) + \hal(c_1) \leq 2$ then $\hal(c_5)$ is equal to $3$ and, by Claim~\ref{claim:145}, $N^{aa}_{15}(\Pi) \leq 18$, and this is a contradiction. So $N^{aa}_{15}(\Pi) = 19, N^{bb}_{15}(\Pi) = 18$ and $N^{cc}_{15}(\Pi) \leq 17$.

Now we suppose that $N^{bb}_{15}(\Pi) \leq 17$. The only case we concern about is when $N^{cc}_{15}(\Pi) = 19$. Since $C$ satisfies the Claim~\ref{claim:145}, then in the moment that $C$ changes with $A$ we will get $N^{aa}_{15}(\Pi) \leq 18$, which is a contradiction. Thus $N^{aa}_{15}(\Pi) = 19, N^{bb}_{15}(\Pi) \leq 17$ and $N^{cc}_{15}(\Pi) \leq 18$.

So, $N_{15}(\Pi)=N_{15}^{mono}(\Pi)+N_{15}^{bi}(\Pi)=69$. This implies that $N_{14}(\Pi) = 75$, and by ($\ref{eq:credg}$) we are done. \qed

\section{Concluding Remarks}

In this paper we have presented a result that relate the number of $(\leq k)$-edges with $3$-decomposability. That is, every set of points in the plane which has a certain number of $(\leq k)$-edges, can be grouped into three independent equal sized sets. Theorem~\ref{thm:main} goes a step forward to the understanding of the structure of sets minimizing the number of $(\leq k)$-edges. Aichholzer et al.~\cite{AiGa} established that such sets always have a triangular convex hull. Here we show that these sets also are $3$-decomposable.

As an application of Theorem~\ref{thm:main}, we give a free computer-assisted proof that the rectilinear crossing number of $K_{30}$ is $9726$. This closes the gap between $9723$ and $9726$, the best lower and upper bounds previously known.

In view of Theorem~\ref{thm:main}, we now give a more precise version of the Conjecture 1 in~\cite{Ab:3sim}:

\begin{conj}\label{conj:1}
For each positive integer $n$ multiple of $3$, all crossing-minimal geometric drawings of $K_n$ have exactly $3\binom{k+2}{2}$ $(\leq k)$-edges for all $0 \le k \le n/3$.
\end{conj}

We believe that Conjecture~\ref{conj:1} is one of the main problems to solve in order to understand the basic structure of the crossing-minimal geometric drawings of $K_n$.

\section{Acknowledgements}

We thank Gelasio Salazar for his help and valuable discussions. We also thank an anonymous referee for his suggestions and recommendations to improve the presentation of this paper.

\end{document}